\documentclass[]{article}
\usepackage{amssymb,amsmath,amsthm,mathrsfs,nicefrac,url}
\usepackage{color,citesort}
\newcommand{\R}{\mathbf{R}}
\newcommand{\e}{\mathrm{e}}
\renewcommand{\P}{\mathrm{P}}

													\usepackage{tikz}	

\newcommand {\E}{\mathrm{E}}

\renewcommand{\d}{\mathrm{d}}

\renewcommand{\phi}{\varphi}
\newcommand{\lip}{\mathrm{Lip}}
\renewcommand{\le}{\leqslant}
\renewcommand{\leq}{\leqslant}
\renewcommand{\ge}{\geqslant}
\renewcommand{\geq}{\geqslant}

\newtheorem{stat}{Statement}[section]
\newtheorem{proposition}[stat]{Proposition}
\newtheorem{corollary}[stat]{Corollary}
\newtheorem{theorem}[stat]{Theorem}
\newtheorem{lemma}[stat]{Lemma}
\theoremstyle{definition} 
\newtheorem{definition}[stat]{Definition}

\numberwithin{equation}{section}\begin{document}

\title{Analysis of the gradient of the solution to a stochastic
	heat equation via fractional Brownian motion%
	\thanks{Research supported in part by NSF grants DMS-1006903 
	and DMS-1307470 (D.K. and P.M.), EPSRC grant EP/J017418/1 (M.F.),
	the Simons Foundation (Simons Visiting Professorship, D.K.), 
	Mathematisches Forschungsinstitut Oberwolfach (D.K.),
	and TU-Dresden (D.K.).}
}
\author{Mohammud Foondun\\Loughborough University
	\and Davar Khoshnevisan\\University of Utah
	\and Pejman Mahboubi\\University of Utah
}

\date{June 11, 2014}%{February 4, 2014}
\maketitle

\begin{abstract}
	Consider the stochastic partial differential equation
	$\partial_t u = Lu+\sigma(u)\xi$, where
	$\xi$ denotes space-time white noise and $L:=-(-\Delta)^{\alpha/2}$
	denotes the fractional Laplace operator of index $\alpha/2\in(\nicefrac12\,,1]$.
	We study the detailed behavior of the approximate spatial gradient
	$u_t(x)-u_t(x-\varepsilon)$ at fixed times $t>0$, as $\varepsilon\downarrow0$.
	We discuss a few applications of this work to the study of the sample functions
	of the solution to the KPZ equation as well.\\
	
	
\end{abstract}
\section{Introduction and main results}
We consider the stochastic partial differential equation
\begin{align}\label{SHE}
	\frac{\partial}{\partial t} u_t(x) =-(-\Delta)^{\alpha/2}
	u_t(x) + \sigma(u_t(x)) \,\xi,
	\tag{SHE}
\end{align}
where $\alpha\in(1\,,2]$ is a fixed ``spatial scaling''
parameter, $-(-\Delta)^{\alpha/2}$ denotes
the fractional Laplacian to the power $\alpha/2$, and
$\xi$ denotes space-time white noise. 
Throughout, we assume that the initial function $u_0$ 
is non random and bounded. 
We assume also that $\sigma$ is Lipschitz continuous; that is, 
$\lip_\sigma<\infty$, where 
\begin{equation}
	\lip_\sigma:=\sup_{-\infty<x<y<\infty}\left|
	\frac{\sigma(y)-\sigma(x)}{y-x}\right|.
\end{equation}
It is well known that, under the present conditions,  (SHE) has
a continuous solution that is uniquely defined by the moment
condition
\begin{equation}\label{eq:moments}
	\sup_{x\in\R}\sup_{t\in[0,T]}\E\left(\left| u_t(x)\right|^k\right)<\infty
	\qquad\text{for all $T>0$},
\end{equation}
for one, hence all, $k\in[2\,,\infty)$; see  \cite{DebbiDozzi}.

The main objective of this paper is to study the spatial gradient of the
random function $x\mapsto u_t(x)$ where the time parameter $t>0$
is held fixed. Our results will have the \emph{a priori} [unsurprising]
consequence that
$\partial u_t(x)/\partial x$ does not exist. Therefore, instead of
studying the gradient itself we consider an approximate un-normalized
gradient $u_t(x)-u_t(x-\varepsilon)$, where $\varepsilon\approx 0$.

Our main result is the following:
\begin{theorem}\label{th:Differentiate}
	For every fixed $t>0$,
	there exists an {\rm fBm} $F:=\{F(x)\}_{x\in\R}$ with Hurst index
	$H:=\frac12(\alpha-1)\in(0\,,\nicefrac12]$ such that 
	\begin{equation}\label{eq:RSDE}
		\lim_{\varepsilon\downarrow0}\sup_{x\in\R}\P\left\{
		\left| \frac{u_t(x) - u_t(x-\varepsilon)}{F(x)-F(x-\varepsilon)} - 
		\mathfrak{A}_\alpha\sigma(u_t(x))
		\right| > \lambda\right\} =0,
	\end{equation}
	for all $\lambda>0$, where $\mathfrak{A}_\alpha$ is the
	following numerical constant:
	\begin{equation}\label{A(alpha)}
		\mathfrak{A}_\alpha := \left\{ 2\Gamma(\alpha)\left\vert
		\cos\left(\alpha\pi/2\right)\right\vert\right\}^{-1/2}.
	\end{equation}
\end{theorem}

Theorem \ref{th:Differentiate} basically asserts that $x\mapsto u_t(x)$
solves the following rough stochastic differential equation:
$$
	\frac{\d u_t(x) }{\d x}
	= \mathfrak{A}_\alpha\sigma(u_t(x))\,
	\frac{\d F(x)}{\d x}\qquad(x\in\R),
	\eqno{\textnormal{(R-SDE)}}
$$
where $F$ denotes an fBm  with Hurst index $H:=\frac12(\alpha-1)\in(0\,,\nicefrac12]$.
It has recently been shown in \cite{KSXZ}, using other methods,
that the \emph{temporal} approximate gradient of 
(SHE) solves a rough differential equation that
is driven by a fractional Brownian motion
with Hurst index $H/\alpha\in(0\,,\nicefrac14]$. 

When $H\in(\nicefrac12\,,1]$, the
random function $F$ is smoother than a H\"older-continuous
function of index $>\nicefrac12$, and hence (R-SDE) can be solved
by using classical theory of Young integrals. Therefore, Theorem
\ref{th:Differentiate} shows the existence of solutions to all
remaining possible rough differential equations that are driven by
fractional Brownian motion. However, we caution that (R-SDE)
is an anticipative stochastic differential equation, even in the case
that $\alpha=2$ where $F$ simplifies to a standard Brownian motion.

Stochastic differential equations
such as (R-SDE) have been the subject of intense recent activity
\cite{AlosLeonNualart,AlosMazetNualart,CoutinFritzVictoir,%
ErramiRusso,FritzVictoir2010a,GradinaruRussoVallois,Gubinelli,KSXZ,%
LyonsQian,Lyons:95,Lyons:94,NualartTindel,RussoVallois,Unterberger}.
As such, Theorem \ref{th:Differentiate} can be viewed
as a non-trivial contribution to the existence theory of very rough 
[anticipative] stochastic 
differential equations. More significantly, the proof of Theorem \ref{th:Differentiate} also 
teaches us about the local structure of
the stochastic heat equation (SHE). For instance, one can conclude fairly easily
from the method we employed in the proof of Theorem \ref{th:Differentiate} a local law of the iterated logarithm, as the following suggests.

\begin{corollary}\label{cor:LIL}
	Choose and fix $t>0$ and $x\in\R$. Then with probability one,
	\[
		\limsup_{\varepsilon\downarrow 0}\frac{u_t(x)-u_t(x-\varepsilon)}{
		\sqrt{2\varepsilon^{\alpha-1}\log\log(1/\varepsilon)}}
		=-\liminf_{\varepsilon\downarrow 0}\frac{u_t(x)-u_t(x-\varepsilon)}{
		\sqrt{2\varepsilon^{\alpha-1}\log\log(1/\varepsilon)}}
		= \mathfrak{A}_\alpha |\sigma(u_t(x))|.
	\]
\end{corollary}

One cannot improve the convergence-in-probability assertion in
\eqref{eq:RSDE} 
to a statement with almost-sure convergence. This is because
there are infinitely-many small random values of $\varepsilon>0$ such that 
the denominator $F(x)-F(x-\varepsilon)$ of the difference quotient in
\eqref{eq:RSDE} is zero. However, the following almost-sure-in-density
improvement does hold:
\begin{corollary}\label{a.s.}
	Fix $t>0$ and $x\in\R$. With probability one,
	\begin{equation}\label{eq:a.s}
		\lim_{s\to0}\frac{1}{s}\int_0^s\left|\frac{u_t(x)-u_t(x-\varepsilon)}
		{F(x)-F(x-\varepsilon)}-\mathfrak{A}_\alpha\sigma(u_t(x))\right|\d\varepsilon
		=0.
	\end{equation}
\end{corollary}

We can also deduce an interesting ``central limit theorem,'' whose limit
law is a mixture of mean-zero normal distributions.

\begin{corollary}\label{cor:CLT}
	Choose and fix $t>0$ and $x\in\R$. Then, for all $a\in\R$,
	\begin{equation}
		\lim_{\varepsilon\downarrow 0}\P\left\{
		\frac{u_t(x)-u_t(x-\varepsilon)}{\varepsilon^{(\alpha-1)/2}} \le a\right\}
		= \P\left\{ |\sigma(u_t(x))| \times
		\mathcal{N}\le \frac{a}{\mathfrak{A}_\alpha}\right\},
	\end{equation}
	where $\mathcal{N}$ denotes a standard Gaussian random variable, independent
	of $u_t(x)$.
\end{corollary}

Let us also mention a corollary about the variations of the solution to (SHE).
This result extends the recent work of Posp\`\i{}\v{s}il and Tribe \cite{PospisilTribe}.

\begin{corollary}\label{cor:VAR}
	If $\phi:\R\to\R$ is Lipschitz continuous, then
	for all non random reals $b>a$ and $t>0$,
	\begin{equation}\begin{split}
		&\lim_{n\to\infty} \sum_{a2^n\le j \le b2^n} \phi( u_t
			(j 2^{-n})) \left| u_t( (j+1)2^{-n})- u_t\left( j2^{-n}\right)
			\right|^{2/(\alpha-1)}\\
		&\hskip1.7in
			=\mathfrak{B}_\alpha \int_a^b \phi(u_t(x))
			\left[ \sigma(u_t(x))\right]^{2/(\alpha-1)}\d x,
	\end{split}\end{equation}
	almost surely and in $L^2(\P)$, where $\mathfrak{B}_\alpha$ is
	the following numerical constant:
	\begin{equation}\label{B(alpha)}
		\mathfrak{B}_\alpha
		:= \frac{1}{\sqrt\pi}\left| \frac{1}{\Gamma(\alpha)\cos(\alpha\pi/2)}
		\right|^{1/(\alpha-1)} \Gamma\left( \frac{1}{2}+\frac{1}{\alpha-1}\right).
	\end{equation} 
\end{corollary}

In Section \ref{sec:VS} we describe further consequences of these
Corollaries to the analysis of the KPZ equation of statistical mechanics. See,
in particular, Corollary \ref{cor:KPZ}. The latter corollary is deeply connected to,
and complements,
the recent works \cite{Hairer,QuastelRemenik} on the quadratic variation, in 
the space variable, of the Hopf--Cole solution to the KPZ equation.

Throughout, we write
\begin{equation}\label{eq:nabla}
	(\nabla_\varepsilon f)(x) := f(x)-f(x-\varepsilon),
\end{equation}
as substitute for the approximate spatial gradient of any real
function $f$. We also adopt the following notation 
\begin{equation}
	\|X\|_k:=\left\{ \E(|X|^k) \right\}^{1/k},
\end{equation}
 for every $k\in [1,\,\infty)$ and $X\in L^k(\Omega)$.
\section{Preliminaries}
Throughout this paper, we make use of the Walsh 
theory of SPDEs \cite{Walsh} to interpret  \eqref{SHE} as the following integral equation, 
\begin{equation}\label{mild}
	u_t(x) = (p_t*u_0)(x) + \int_{(0,t)\times\R} p_{t-s}(y-x)\sigma(u_s(y))\, \xi(\d s\,\d y),
\end{equation}
where $\{p_t(x)\}_{t>0,x\in\R}$ denotes the heat kernel of the fractional Laplacian
and the stochastic integral is a Walsh--It\^o integral. 
The function $(t\,,x)\mapsto p_t(x)$ 
also describes the transition density functions 
of an isotropic $\alpha$-stable L\'evy
process $\{X_t\}_{t\ge 0}$, 
and are determined via their Fourier transforms that are normalized as follows:
\begin{equation}\label{pihat}
	\hat p_t(\chi)=\e^{-t|\chi|^\alpha}\qquad(t> 0,\, \chi\in\R).
\end{equation} 

The theory of Dalang \cite{Dalang} ensures the existence of a 
solution $u$ to the SPDE (SHE), which is unique among all
solutions that satisfy  \eqref{eq:moments} 
for all $T>0$ and $k\in[2\,,\infty)$.

The proof of the main result of this paper will rely on some perturbation
arguments which will require the study of the following 
linear stochastic heat equation.  The main results of this section 
will be devoted to the study of this linear equation.  First, let us 
fix some notation. 

Let us consider the following linearization of (SHE):
\begin{align}\label{SHE-Linear}
	\frac{\partial}{\partial t} Z_t(x) =-(-\Delta)^{\alpha/2}
	Z_t(x) + \xi,
	\tag{L-SHE}
\end{align}
subject to $Z_0(x):=0$ for all $x\in\R$.
In keeping with \eqref{mild},
the solution to (L-SHE) can be written as the Wiener-integral process,
\begin{equation}\label{Z}
	Z_t(x) := \int_{(0,t)\times\R} p_{t-s}(y-x)\, \xi(\d s\,\d y)
	\qquad(t>0,\, x\in\R).
\end{equation}

It is very well known that $x\mapsto Z_t(x)$ has a version that is
H\"older continuous of any index $<(\alpha-1)/2$. This
fact relies on another well-known bound of the form
$\E(|Z_t(x)-Z_t(x-\varepsilon)|^2)=O(\varepsilon^{\alpha-1})$,
as $\varepsilon\downarrow 0$.
Our first result is basically an improvement of such an estimate. 
It will be of central importance to our later needs.

\begin{lemma}\label{(0,t)timesR}
	Choose and fix a $T>0$. Then, 
	\begin{equation}
		\E\left( \left| Z_t(x) - Z_t(x-\varepsilon)\right|^2\right)
		=\mathfrak{A}_\alpha^2\varepsilon^{\alpha-1}
		+ O(\varepsilon^2)
		\qquad(\varepsilon\downarrow 0),
	\end{equation}
	uniformly for all $t\in[0\,,T]$ and $x\in\R$, where $\mathfrak{A}_\alpha$ is defined in \eqref{A(alpha)}.
\end{lemma}

\begin{proof}
	To aid the presentation of the proof, we define
	\begin{equation}
		Q := \E\left(\left| Z_t(x) - Z_t(x-\varepsilon) \right|^2\right)
		\qquad(t,\varepsilon>0,\,x\in\R).
	\end{equation}
	The mild formulation of the solution together with Wiener's isometry and Plancherel's formula yield
	\begin{equation}\label{D2}\begin{split}
		Q &=\int_0^t\d s\int_{-\infty}^\infty\d y\
			\left[ p_s(y)-p_s(y-\varepsilon)\right]^2\\
		&= \frac{1}{2\pi}\int_0^t\d s\int_{-\infty}^\infty\d\chi\
			\e^{-2s|\chi|^\alpha}\left| 1- \e^{-i\chi\varepsilon}\right|^2\\
		&=\frac{1}{\pi}\int_0^t\d s\int_{-\infty}^\infty\d\chi\
			\e^{-2s|\chi|^\alpha}[1-\cos(\chi\varepsilon)],
	\end{split}\end{equation}
	where we have also used the simple fact that $|1-\exp(i\theta)|^2=2(1-\cos\theta)$
	for every $\theta\in\R$. We now use Fubini's theorem in order
	to interchange the order of integration, and then use symmetry to deduce
	the following:
	\begin{equation}\label{D3}\begin{split}
		Q & = \frac1\pi
			\int_0^\infty\frac{1-\e^{-2t|\chi|^\alpha}}{\chi^\alpha}
			[1-\cos(\chi\varepsilon)]\,\d\chi\\
		&=\frac1\pi\int_0^\infty\frac{1-\cos(\chi\varepsilon)}{\chi^\alpha}\,\d\chi
			-\frac1\pi\int_0^\infty\frac{\e^{-2t|\chi|^\alpha}}{\chi^\alpha}
			[1-\cos(\chi\varepsilon)]\,\d\chi.
	\end{split}\end{equation}
	Next, we separately compute the preceding integrals.  
	In order to calculate the first, we use scaling and then 
	evaluate the resulting integral using \cite[\S4, p.\ 13]{K-SN}:
	\begin{equation}\label{eq:1-cos}
		\int_0^\infty\frac{1-\cos(\chi\varepsilon)}{\chi^\alpha}\,\d\chi
		=\varepsilon^{\alpha-1} \cdot\int_0^\infty\frac{1-\cos z}{z^\alpha}\,\d z
		=\pi \mathfrak{A}_\alpha^2\varepsilon^{\alpha-1}.	
	\end{equation}
	Since $1-\cos \theta \le \theta^2$ for all $\theta\in\R$,
	the second integral on the last line of \eqref{D3} is bounded from above by
	\begin{equation}\label{eq:e(1-cos)}
		\int_0^\infty\frac{\e^{-2t|\chi|^\alpha}}{\chi^\alpha}
		[1-\cos(\chi\varepsilon)]\,\d\chi \le \varepsilon^2\cdot
		\int_0^\infty \e^{-2t\chi^\alpha}\chi^{2-\alpha}\,\d\chi=O(\varepsilon^2),
	\end{equation}
	uniformly for all $t\in[0\,,T]$. The lemma follows from combining the
	preceding two displays with \eqref{D3}.
\end{proof}

Lemma \ref{(0,t)timesR} has a number of immediate, though still useful,
consequences. We list two of those that we shall need. The
first computes the correct power that ``essentially linearizes''
the spatial increments of the process $Z$.

\begin{corollary}\label{cor:Mean:Linear}
	Choose and fix a $T>0$. Then, 
	\begin{equation}
		\E\left( \left| Z_t(x) - Z_t(x-\varepsilon)\right|^{2/(\alpha-1)}\right)
		=\mathfrak{B}_\alpha\varepsilon + O\left(
		\varepsilon^{2/(\alpha-1)} \right) \qquad(\varepsilon\downarrow 0),
	\end{equation}
	uniformly for all $t\in[0\,,T]$ and $x\in\R$.
\end{corollary}

\begin{proof}
	We appeal to the fact that every positive moment of a 
	centered Gaussian random variable $X$ is determined by the variance
	of $X$.  In particular,
		\begin{equation}\label{theabove}
			\E\left(|X|^{2/(\alpha-1)}\right) =
			c \left[\E\left(X^2\right)\right]^{1/(\alpha-1)},
		\end{equation}
		where
		\begin{equation}
			c  := \frac{1}{\sqrt{2\pi}}\int_{-\infty}^\infty|z|^{2/(\alpha-1)}
			\e^{-z^2/2}\,\d z = \frac{2^{1/(\alpha-1)}}{\sqrt\pi}\Gamma\left(
			\frac12+\frac{1}{\alpha-1}\right).
		\end{equation}
	For every fixed $t\in [0\,, T]$ and $x\in \R$,  
	the increment $Z_t(x) - Z_t(x-\varepsilon)$ is a Gaussian random variable with
	mean zero.  
	Therefore we may apply \eqref{theabove} 
	[with $X:=Z_t(x) - Z_t(x-\varepsilon)$],
	together with Lemma \ref{(0,t)timesR} and \eqref{B(alpha)},
	in order to conclude the proof of the corollary.
\end{proof}

Next we mention the following consequence of Lemma 
\ref{(0,t)timesR}; it states the increments of the solution $Z$ behaves much like those of the solution $u$ to the non-linear (SHE).

\begin{corollary}\label{cor:Modulus:x}
	Let $u_t(x)$ denote the solution to \eqref{SHE}. Then, for all real numbers
	$T>0$ and $k\in[2\,,\infty)$
	there exists a finite constant $C:=C_{k,T}$, such that 
	\begin{equation}
		\E\left( \left| u_t(x)-u_t(y)
		\right|^k\right)\le C|x-y|^{(\alpha-1)k/2},
	\end{equation}
	simultaneously for every $x,y\in\R$ and $t\in[0,T]$.
\end{corollary}

\begin{proof}
	Thanks to \eqref{eq:moments}, it suffices to consider the case that
	$|x-y|\le 1$, which we do from here on.
	
	Since $u_0$ is bounded, $(p_t*u_0)(x)$ is smooth for every fixed $t>0$. 
	Indeed, \eqref{pihat} ensures that
	$\hat p_t(\chi)$---whence also $p_t$---are in the Schwartz space
	of rapidly decreasing test functions.  Because $u_0$ is a tempered distribution
	a standard fact  \cite[Theorem 3.13]{S-W} implies that
	$(p_t*u_0)(x)$ is smooth. 
	An application of
	a sharp form of the Burkholder--Davis--Gundy inequality---as in \cite{FK}---%
	ensures that, uniformly for all
	$x,y\in\R$ and $t\in[0\,,T]$,
	\begin{align}
		&\|u_t(x)-u_t(y)\|^2_k\\\notag
		&\le c_2\left[|x-y|^2
			+\int_0^t\d s\int_{-\infty}^\infty\d z\
			|p_{t-s}(z-x)-p_{t-s}(z-y)|^2\|\sigma(u_s(z))\|_k^2\right],
	\end{align}
	where $c_2$ is positive and finite, and depends only on $k$ and $T$.
	[See \cite{FK} for the details of this sort
	of argument.]
	
	Because $\sigma$ has at-most-linear growth, \eqref{eq:moments}
	ensures that $\|\sigma(u_s(y))\|_k^2$ is bounded uniformly in
	$s\in[0\,,T]$ and $y\in\R$, and therefore we can find
	$c_3:=c_{k,T}\in(0\,,\infty)$ such that, uniformly for $x,y\in\R$ and
	$t\in[0\,,T]$,
	\begin{align}\notag
		\|u_t(x)-u_t(y)\|^2_k
			&\le c_3\left[|x-y|^2
			+\int_0^t\d r\int_{-\infty}^\infty\d z\
			|p_r(z-x)-p_r(z-y)|^2\right]\\
		&=c_3\left[|y-x|^2+ \E\left( \left| Z_t(x)-Z_t(y) \right|^2\right)\right].
	\end{align}
	Because we consider only the case that $|x-y|\le1$,
	Lemma \ref{(0,t)timesR} shows us that
	$\|u_t(x)-u_t(y)\|^2_k\le\text{const}\cdot |x-y|^{\alpha-1}$, as desired.
\end{proof}

Let us conclude this section with another continuity estimate. 
This bound follows essentially from the Appendix of \cite{FK};
see also \cite{DebbiDozzi}. Therefore, we will not describe a proof.
\begin{proposition}\label{pro:Modulus:t}
	Suppose that $u_t(x)$ be as in the previous Corollary, then for all $T>0$
	and $k\in[2\,,\infty)$, there exists a finite constant $C:=C_{k, T}$ such that 
	\begin{equation}
		\|u_s(x)-u_t(x)\|_k\leq C|t-s|^{(\alpha-1)/2\alpha},
	\end{equation}
	simultaneously for every $s, t\in [0, T]$ and $x\in \R$.
\end{proposition}

\section{Fractional Brownian motion}
Recall that a random field
$\Phi:=\{\Phi(x)\}_{x\in\R}$ is called a \emph{fractional
Brownian motion with Hurst index $H\in(0\,,1)$} if $\Phi$ is a mean-zero
Gaussian process with 
\begin{equation}
	\E\left( \left| \Phi(x) - \Phi(y) \right|^2\right) 
	= |x-y|^{2H}\qquad\text{for all $x,y\in\R$}.
\end{equation}
We abbreviate this by saying that $\Phi$ is a $\textnormal{fBm}(H)$.
In particular, fBm($\nicefrac12$) is easily seen to be ordinary Brownian motion.

\begin{proposition}\label{pr:fBm}
	Fix some $t>0$. Then, the solution $Z_t$ to the linear stochastic  heat equation
	\textnormal{(L-SHE)}, at time $t$, can be decomposed as 
	\begin{equation}\label{eq:ZFS}
		Z_t(x) = \mathfrak{A}_\alpha
		F(x) + S(x)\qquad\text{for all $x\in\R$},
	\end{equation}
	where $F$ is $\textnormal{fBm}((\alpha-1)/2)$ and $S$ is a centered
	Gaussian process with $C^\infty$ sample functions.
\end{proposition}

\begin{proof}
	The random field $S$ is defined explicitly as the following
	Wiener integral process:
	\begin{equation}
		S(x) := 
		\int_{(t,\infty)\times\R} \left[ p_s(y)-p_s(y-x)\right] \, \xi(\d s\,\d y)
		\qquad(x\in\R).
	\end{equation}
	In order to verify that $S$ is a well-defined Gaussian random field
	we proceed by applying the Plancherel theorem, such as we did when
	we developed \eqref{D2}, and find that
	\begin{equation}\begin{split}
		\E\left( \left| S(x) \right|^2\right) &= \int_t^\infty \d s\int_{-\infty}^\infty
			\d y\ \left| p_s(y) - p_s(y-x)\right|^2\\
		&=\frac{2}{\pi}\int_t^\infty\d s\int_0^\infty\d\chi\
			\e^{-2s\chi^\alpha}\left[ 1-\cos(x\chi)\right]\\
		&=\frac{1}{\pi}\int_0^\infty \e^{-2t\chi^\alpha}
			\left(\frac{1-\cos(x\chi)}{\chi^\alpha}\right)\d\chi<\infty.
	\end{split}\end{equation}
	This proves that $S$ is a mean-0 Gaussian random field. Similarly,
	one shows that its $n$th generalized
	derivative $\d^n S/\d x^n$ is the mean-0
	Gaussian random field 
	\begin{equation}
		\frac{\d ^n S}{\d x^n}(x) = (-1)^{n+1}\int_{(t,\infty)\times\R}
		\frac{\d ^n p_s}{\d x^n} (y-x)\, \xi(\d s\,\d y).
	\end{equation}
	And we verify that the preceding is a well-defined bona fide stochastic
	process by checking that
	\begin{equation}\begin{split}
		\E\left( \left| \frac{\d^n S}{\d x^n}(x)\right|^2\right)
			&= 	\int_t^\infty\d s\int_{-\infty}^\infty\d y\
			\left|\frac{\d ^n p_s}{\d x^n} (y)\right|^2\\
		&=\frac{1}{2\pi}\int_t^\infty\d s\int_{-\infty}^\infty\d\chi\
			|\chi|^{2n}\e^{-2s|\chi|^\alpha},
	\end{split}\end{equation}
	thanks to Plancherel's theorem. Therefore, for every $x\in\R$
	and $n\ge 1$,
	\begin{equation}\label{eq:E(dS2)}
		\E\left( \left| \frac{\d^n S}{\d x^n}(x)\right|^2\right)
		= \frac{1}{2\pi}\int_0^\infty
		\e^{-2t\chi^\alpha}\chi^{2n-\alpha}\,\d\chi<\infty.
	\end{equation}
	The remaining details involve making a few routine
	computations that are similar to some of the calculations that were
	made earlier in this section. We only discuss the case when $n=1$.
	Since, by the Wiener isometry, for any $x_2>x_1$,
	\begin{align}
		\E\left(\left|\frac{\d S}{\d x}(x_2)-\frac{\d S}{\d x}(x_1)\right|^2\right)
		&=\int_t^\infty\d s\int_\R\d y\,|p'_s(y-x_2)-p'_s(y-x_1)|^2\notag\\
		%&=\frac{1}{\pi}\int_t^\infty\d s\int_\R\d\chi\,|\chi\,\hat p_s(\chi)|^2
		%[1-\cos\chi(x_2-x_1)]\\
		&=\frac{1}{\pi}\int_0^\infty\e^{-2t|\chi|^\alpha}
		\frac{1-\cos\chi(x_2-x_1)}{\chi^{\alpha-1}}\d\chi\\
		&=O(|x_2-x_1|^2),\notag
	\end{align}
	 by~\eqref{eq:e(1-cos)}. Therefore, by the Kolmogorov's continuity theorem, 
	 the distributional derivative $(\d S/\d x)$ of 
	 $S$ is a continuous  function [up to a modification, which we may
	 choose to adopt without further mention], whence
	$S\in C^1(\R)$ a.s. And a very similar estimate shows that
	if $S\in C^k(\R)$ for some integer $k\ge 1$, then
	$S\in C^{k+1}(\R)$ a.s.\ [up to a modification, once again]. This
	proves that $S$ is a.s.\ $C^\infty$. 
	
	Now that we have defined $S$, let us observe that $Z_t$ and $S$ are
	totally independent from one another [this has to do only with the independence
	properties of white noise]. Therefore, it follows
	that 
	\begin{equation}\label{eq:B=Z-S}
		B(x) := Z_t(x)-S(x) \qquad(x\in\R)
	\end{equation}
	defines a mean-zero Gaussian random field
	whose distribution is computed as follows: For every $x,x'\in\R$,
	\begin{equation}\begin{split}
		\E\left( \left| B(x)-B(x')\right|^2\right) &= 
			\E\left( \left| Z_t(x)-Z_t(x') \right|^2\right) + 
			\E\left( \left| S(x)-S(x') \right|^2\right)\\
		&=\int_0^\infty\d s\int_{-\infty}^\infty\d y\ 
			\left[ p_s(y-x) - p_s(y-x')\right]^2\\
		&= \frac{1}{\pi}\int_0^\infty\d s\int_{-\infty}^\infty\d\chi\
			\e^{-2s|\chi|^\alpha}\left[ 1-\cos((x-x')\chi)\right],
	\end{split}\end{equation}
	thanks to Plancherel's theorem. This and Fubini's theorem together yield
	\begin{equation}\begin{split}
		\E\left( \left| B(x)-B(x')\right|^2\right) 
			&= \frac1{\pi}\int_0^\infty
			\left(\frac{ 1-\cos(|x-x'|\chi)}{\chi^\alpha}\right)\d\chi\\
		&=\pi^{-1}|x-x'|^{\alpha-1}\int_0^\infty
			\left(\frac{ 1-\cos z}{z^\alpha}\right)\d z\\
		&= \frac{|x-x'|^{\alpha-1}}{2\Gamma(\alpha)\vert\cos(\alpha\pi/2)\vert}
			=[\mathfrak{A}_\alpha]^2|x-x'|^{\alpha-1};
	\end{split}\end{equation}
	see \eqref{A(alpha)} and \eqref{eq:1-cos}. In particular,
	$F(x) := B(x)/\mathfrak{A}_\alpha$ $(x\in\R)$
	defines a fractional Brownian motion with Hurst index
	$(\alpha-1)/2$; and property \eqref{eq:ZFS} follows from the
	construction in \eqref{eq:B=Z-S}.
\end{proof}

\section{Localization}
Throughout, let us choose and fix a parameter
\begin{equation}\label{eq:beta:gamma}
	\beta>1,\text{ and set }\gamma := 1 + \beta^{3/2}.
\end{equation}
Then, we define a family of  space-time boxes as follows: For every $x\in\R$
and $\varepsilon>0$,
\begin{equation}\label{B}
	\mathbf{B}_\beta(x\,,t\,;\varepsilon) := \left[ t-\beta\varepsilon^\alpha,t \right]\times 
	[x-\varepsilon\gamma\,,x+\varepsilon\gamma].
\end{equation}
When $\varepsilon\approx0$, the preceding describes a very small box
in space-time; see the figure below.

\begin{figure}[h!]
 \begin{tikzpicture}[line/.style={&gt;=latex}]
\fill [white] (14.5,-.26) coordinate (0,1) rectangle (21,1.4);    
\fill [orange!30] (17.57,1) coordinate (0,1) rectangle(19.47,1.37);                 
   \draw [->] (17,-1.5) -- (17,2) node [right] {};
      \draw [white, thick] (17,-1.5) -- (17,-.5) node [right] {};
            \draw [white, thick] (17,-1.5) -- (17,-.5) node [right] {};
      \draw  [->](13,-.27) -- (21,-.27) node [above right] {};
            \draw  [->, white, very thick](13,-.27) -- (14.5,-.27) node [above right] {};
              \draw  [->, white, thick](13,-.27) -- (14.5,-.27) node [above right] {};
            \draw  [dashed](17,1.4) -- (19.5,1.4) node [right] {};
\draw  [dashed](17,1) -- (19.5,1) node [right] {};
             \draw [dashed] (17.55,-.2) -- (17.55,1.3) node [right] {};
               \draw [dashed] (18.55,-.2) -- (18.55,1) node [right] {};
        \draw [dashed] (19.5,-.2) -- (19.5,1.3) node [right] {};
\draw(16.6,1.4) node {$t$};
\draw(16.4,1) node {$t-\beta\varepsilon^\alpha$};
\draw(17.6,-.6) node {$x-\varepsilon\gamma$};
\draw(19.6,-.6) node {$x+\varepsilon\gamma$};
\draw(18.55,1.4) node {$\bullet$};
\draw[](18.55,1.8) node {$(t\,,x)$};
\draw[](18.55,1.19) node {\textcolor{darkgray}{${\mathbf B}_\beta(x\,,t\,;\varepsilon)$}};
\end{tikzpicture}
\caption{A depiction of the localization region}
\end{figure}
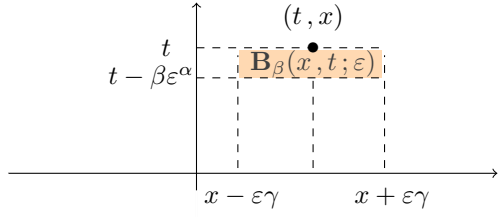

Let us recall the solution $Z$ to the linear equation (L-SHE), and 
the [approximate] gradient operator $\nabla_\varepsilon$ from \eqref{eq:nabla}.
We may observe that
\begin{equation}\label{eq:Z:SI}
	(\nabla_\varepsilon Z_t)(x) = Z_t(x) - Z_t(x-\varepsilon)=
	\int_{(0,t)\times\R}
	(\nabla_\varepsilon p_{t-s})(y-x)\, \xi(\d s\,\d y).
\end{equation}
The following is the main technical computation of this section.

\begin{proposition}[Localization of the gradient]\label{pr:localize:Z}
	There exists a finite and positive constant $A$---depending only on the
	value of $\alpha$---such that
	\begin{equation}
		\E\left( \left| (\nabla_\varepsilon Z_t)(x) - \int_{\mathbf{B}_\beta(x,t;\varepsilon) }
		(\nabla_\varepsilon p_{t-s})(y-x)\, \xi(\d s\,\d y)\right|^2\right) \le
		A\varepsilon^{\alpha-1}\beta^{-1/2},
	\end{equation}
	simultaneously for all $x\in\R$, $t>0$,
	$\varepsilon\in(0\,,1)$ and $\beta>1$.
\end{proposition}

In the case that $\alpha=2$,
the preceding is similar to Lemma 3.6 of Hairer, Maas, and Weber \cite{HMW},
which has been used as a core of the solution theory of rough Burgers-like equations
\cite{H,Hairer,HM}. The novelty here is
the sharp description of the localization estimates in terms of the auxilliary parameter $\beta$.

\begin{proof}
	The Wiener isometry allows us to write
	\begin{equation}\begin{split}	
		Q:=&\ \E\left(\left| (\nabla_\varepsilon Z_t)(x) - \int_{\mathbf{B}_\beta(x,t;\varepsilon) }
			(\nabla_\varepsilon p_{t-s})(y-x)\, \xi(\d s\,\d y)\right|^2\right)\\
		=&\ \int_{[(0,t)\times\R]\setminus \mathbf{B}_\beta(0,t;\varepsilon)} 
			\left| (\nabla_\varepsilon p_{t-s})(y)
			\right|^2\, \d s\,\d y.
	\end{split}\end{equation}
	Let us  decompose $Q$ as
	\begin{equation}
		Q = Q_1 + Q_2,
	\end{equation}
	where
	\begin{equation}
		Q_1 := \int_0^{t-\beta\varepsilon^\alpha}\d s\int_{-\infty}^\infty
		\d y \left| \left(\nabla_\varepsilon p_{t-s}\right)(y)\right|^2,
	\end{equation}
	and
	\begin{equation}
		Q_2 := \int_{t-\beta\varepsilon^\alpha}^t\d s\int_{|y|>\varepsilon\gamma}
		\d y \left| \left(\nabla_\varepsilon p_{t-s}\right)(y)\right|^2.
	\end{equation}
	We estimate $Q_1$ and $Q_2$ in this order, since it is easier to bound $Q_1$.
	
	Because  $({\nabla_\varepsilon p_s})^{\widehat{}}\, (\chi) = \e^{-s|\chi|^\alpha} 
	(1-\e^{-i\chi\varepsilon}),$ Plancherel's theorem implies that
	\begin{align}\notag
		Q_1 &= \frac{1}{\pi}\int_{\beta\varepsilon^\alpha}^t\d s
			\int_{-\infty}^\infty\d\chi\ \e^{-2s|\chi|^\alpha}
			\left[ 1-\cos(\chi\varepsilon)\right]
			\le\frac{\varepsilon^2}{\pi}\int_{\beta\varepsilon^\alpha}^\infty\d s
			\int_{-\infty}^\infty\d\chi\ \e^{-2s|\chi|^\alpha}\chi^2\\
		&=\frac{\varepsilon^2}{2\pi} 
			\int_{-\infty}^\infty\e^{-2\beta\varepsilon^\alpha|\chi|^\alpha}
			|\chi|^{2-\alpha}\,\d\chi
			= \frac{\varepsilon^{\alpha-1}(2\beta)^{-(3-\alpha)/\alpha}}{\pi\alpha}
			\,\Gamma\left(\frac{3-\alpha}{\alpha}\right).
	\end{align}
	This is the desired bound for $Q_1$. The estimation of  $Q_2$ requires 
	a little more effort.
	
	First, we write
	\begin{align}
		Q_2 &=\int_0^{\beta\varepsilon^\alpha} \d s\int_{|y|>\varepsilon\gamma}
			\d y \left| (\nabla_\varepsilon p_s)(y)\right|^2\\\notag
		&\le 2\int_0^{\beta\varepsilon^\alpha} \d s\int_{|y|>\varepsilon\gamma}
			\d y \left| p_s(y)\right|^2 + 
			2\int_0^{\beta\varepsilon^\alpha} \d s\int_{|y|>\varepsilon\gamma}
			\d y \left| p_s(y-\varepsilon)\right|^2,
	\end{align}
	
	Next, let us recall that the inversion formula and symmetry together
	imply that
	$p_s(z) := \pi^{-1}\int_0^\infty\cos(z\chi)\exp\{-s\chi^\alpha\}\,\d\chi$,
	whence
	\begin{equation}\label{eq:p}
		\sup_{z\in\R}p_s(z) = p_s(0) 
		=\frac{\Gamma(1/\alpha)}{\alpha}\, s^{-1/\alpha}.
	\end{equation}
	In particular, if $0\le \eta\le1$ then
	\begin{equation}
		\int_{|y|>\varepsilon\gamma} |p_s(y-\eta\varepsilon)|^2\,\d y
		\le\frac{\Gamma(1/\alpha)}{\alpha}\, s^{-1/\alpha}
		\cdot\P\{|X_s|>\varepsilon(\gamma-\eta)\}.
	\end{equation}
	Since $\gamma=1+\beta^{3/2}$ [see \eqref{eq:beta:gamma}],
	we apply the triangle inequality to see that
	\begin{equation}
		\sup_{0\le\eta\le1}
		\int_{|y|>\varepsilon\gamma} |p_s(y-\eta\varepsilon)|^2\,\d y
		\le\frac{\Gamma(1/\alpha)}{\alpha}\, s^{-1/\alpha}
		\cdot\P\{|X_s|>\varepsilon\beta^{3/2}\}.
	\end{equation}
	By scaling, $\P\{|X_s|>\varepsilon\beta^{3/2}\}=
	\P\{|X_1|>\varepsilon\beta^{3/2}s^{-1/\alpha}\}$;
	and a well-known bound
	 on the tail of stable distributions shows that
	 $\P\{|X_1|>\lambda\}\le\text{const}\cdot\lambda^{-\alpha}$
	 for all $\lambda>0$, and hence
	 \begin{equation}
	 	\sup_{0\le\eta\le1}
		\int_{|y|>\varepsilon\gamma} |p_s(y-\eta\varepsilon)|^2\,\d y
		\le\text{const}\cdot s^{(\alpha-1)/\alpha} \varepsilon^{-\alpha}\beta^{-3\alpha/2}.
	 \end{equation}
	 We integrate this quantity from $s=0$ to $s=\beta\varepsilon^\alpha$
	 in order to see that
	 \begin{equation}
	 	Q_2 \le\text{const}\cdot\varepsilon^{\alpha-1}\beta^{2-(1/\alpha)-(3\alpha/2)}.
	 \end{equation}
	 
	 Finally, we can combine our bounds for $Q_1$ and $Q_2$
	 in order to deduce the inequality,
	 \begin{equation}
	 	Q \le \text{const}\cdot\varepsilon^{\alpha-1}\left[ \beta^{-(3-\alpha)/\alpha}
		+\beta^{2-(1/\alpha)-(3\alpha/2)}\right].
	 \end{equation}
	 Elementary analysis of the exponent of $\beta$ shows that
	 \begin{equation}
	 	\min\left( \frac{3-\alpha}{\alpha} \,, -2+\frac{1}{\alpha}+\frac{3\alpha}{2}\right) 
		\ge\frac12,
	 \end{equation}
	 because $1<\alpha\le 2$. Since $\beta>1$, it follows that 
	 $Q\le \text{const}\cdot \varepsilon^{\alpha-1}\beta^{-1/2}$,
	 as desired.
\end{proof}
Since Lemma \ref{(0,t)timesR} says that
$\E(|(\nabla_\varepsilon Z_t)(x)|^2)\approx \text{const}\cdot\varepsilon^{\alpha-1}$
when $\varepsilon\ll1$,
Proposition \ref{pr:localize:Z} shows that
\begin{equation}
	(\nabla_\varepsilon Z_t)(x) \approx \int_{\mathbf{B}_\beta(x\,,t\,;\varepsilon) }
	(\nabla_\varepsilon p_{t-s})(y-x)\, \xi(\d s\,\d y)
	\ \text{when $\beta\gg1$ and $\varepsilon\ll1$},
\end{equation}
to first order, where the approximation holds in $L^2(\P)$ and
has good uniformity in $(t\,,x\,,\varepsilon\,,\beta)$. 
A quick glance at \eqref{eq:Z:SI} then reveals the meaning of
Proposition \ref{pr:localize:Z}: If $\beta\gg1$ and $\varepsilon\ll1$, 
then most of the contribution to
the stochastic integral in \eqref{eq:Z:SI} comes from the region
$\mathbf{B}_\beta(x\,,t\,;\varepsilon) $. Because
$\mathbf{B}_\beta(x\,,t\,;\varepsilon) $ is a very small subset of $(0\,,t)\times\R$,
Proposition \ref{pr:localize:Z} is describing a \emph{strong
localization} property of the spatial gradient of $u$. 
We will prove that the subsequent results of this paper are
consequences of this localization property.
Let us illustrate the potential usefulness of Proposition \ref{pr:localize:Z}
by showing how it immediately implies a strong localization result for the 
spatial gradient of the solution
$u$ to the fully nonlinear heat equation (SHE).

\begin{corollary}\label{cor:localize:u}
	Choose and fix real numbers $T>0$ and $k\in[2\,,\infty)$. Then,
	there exists a finite constant $A$---depending only on $(\alpha\,,\beta\,,T,k)$---such
	that
	\begin{align}\notag
		&\E\left( \left| (\nabla_\varepsilon u_t)(x)
		 	- \int_{\mathbf{B}_\beta(x,t;\varepsilon) }(\nabla_\varepsilon p_{t-s})(y-x)
			\sigma(u_s(y))\, \xi(\d s\,\d y)\right|^k\right)\\
		&\hskip3in\le A\varepsilon^{(\alpha-1)k/2}\beta^{-k/4},
	\end{align}
	simultaneously for all $x\in\R$, $\varepsilon\in(0\,,1)$, $\beta>1$,
	and $t\in(0\,,T]$.
\end{corollary}

\begin{proof}

We begin by noting that 
	\begin{equation}\begin{split}
		&(\nabla_\varepsilon u_t)(x)\\ 
		&=[\nabla_\varepsilon (p_{t-s}*u_0 )](x)+
			\int_{(0,t)\times\R}(\nabla_\varepsilon p_{t-s})(y-x)
			\sigma(u_s(y))\, \xi(\d s\,\d y).
	\end{split}\end{equation}

	Next we observe that  $| ( \nabla_\varepsilon p_{t-s}*u_0 )(x) |
		\le \text{const}\cdot\varepsilon$; see the discussion at the 
		beginning of the proof of Corollary~\ref{cor:localize:u}.
	Therefore, because $\beta>1$, it suffices to prove that
	\begin{align}\notag
		&\E\left( \left| (\nabla_\varepsilon u_t)(x) -\left(\nabla_\varepsilon p_t*u_0
			\right)(x) - \int_{\mathbf{B}_\beta(x,t;\varepsilon) }(\nabla_\varepsilon p_{t-s})(y-x)
			\sigma(u_s(y))\, \xi(\d s\,\d y)\right|^k\right)\\
		&\hskip3in\le A\varepsilon^{(\alpha-1)k/2}\beta^{-k/2}.
	\end{align}
	Define
	\begin{equation}\begin{split}
		I_1 &:= \int_{(0,t)\times\R}(\nabla_\varepsilon p_{t-s})(y-x)
			\sigma(u_s(y))\, \xi(\d s\,\d y),\\
		I_2 &:= \int_{\mathbf{B}_\beta(x,t;\varepsilon) }(\nabla_\varepsilon p_{t-s})(y-x)
			\sigma(u_s(y))\, \xi(\d s\,\d y).
	\end{split}\end{equation}
	It remains to prove that
	\begin{equation}\label{goal:localize:u}
		\| I_1-I_2\|_k \le \text{const}\cdot \varepsilon^{(\alpha-1)/2}\beta^{-1/2}.
	\end{equation}
	
	An application of the Burkholder--Davis--Gundy inequality shows that
	\begin{equation}
		\|I_1-I_2\|_k^2 \le c\int_{[(0,t)\times\R]\setminus\mathbf{B}_\beta(x,t;\varepsilon) }
		\left| (\nabla_\varepsilon p_{t-s})(y-x)\right|^2
		\left\| \sigma(u_s(y))\right\|_k^2\, \d s\,\d y,
	\end{equation}
	uniformly for all $x\in\R$, $\beta>1$, $\varepsilon\in(0\,,1)$,
	and $t\in[0\,,T]$, where $c$ depends only on $k$ and $T$.
	[See Foondun and Khoshnevisan \cite{FK}
	for the details of this sort of argument.] 
	Thanks to \eqref{eq:moments} and the Lipschitz continuity
	of the function $\sigma$,  $C_k:=\sup_{s\in[0,T]}\sup_{y\in\R}
	\E(|\sigma(u_s(y))|^k)$ is finite. Consequently,
	\begin{equation}
		\E\left(\left| I_1-I_2\right|^k\right)
		\le c^{k/2}C_k\mathcal{T}^{k/2},
	\end{equation}
	where
	\begin{equation}\begin{split}
		\mathcal{T} &:=
			\int_{[(0,t)\times\R]\setminus\mathbf{B}_\beta(x,t;\varepsilon) }
			\left| (\nabla_\varepsilon p_{t-s})(y-x)\right|^2\,\d s\,\d y\\
		&=\E\left(\left|(\nabla_\varepsilon Z_t)(x) - \int_{\mathbf{B}_\beta(x,t;\varepsilon) }
			(\nabla_\varepsilon p_{t-s})(y-x)\, \xi(\d s\,\d y)\right|^2\right).
	\end{split}\end{equation}
	Therefore, Proposition \ref{pr:localize:Z} implies \eqref{goal:localize:u},
	whence the corollary.
\end{proof}

\section{Proof of Theorem \ref{th:Differentiate}}
Our proof of Theorem \ref{th:Differentiate} requires only one
more technical result.

\begin{lemma}\label{lem:localize:u:1}
	Choose and fix $T>0$ and $k\in[2\,,\infty)$. Then there exists a finite
	constant $A$ such that
	\begin{equation}\begin{split}\label{not-predictable}
		&\E\left(\left| \int_{\mathbf{B}_\beta(x,t;\varepsilon) }
			(\nabla_\varepsilon p_{t-s})(y-x)\left[ \sigma(u_s(y))-\sigma(u_t(\tilde x))\right]
			\xi(\d s\,\d y)\right|^k\right) \\
		&\hskip2.7in\le A\varepsilon^{(\alpha-1)k}\beta^{3(\alpha-1)k/4},
	\end{split}\end{equation}
	simultaneously for all $x\in\R$, 
	$\tilde x\in[x-\gamma\varepsilon\,, x+\gamma\varepsilon],$ 
	$t\in[0\,,T]$, $\varepsilon\in(0\,,1)$,
	and $\beta>1$.
\end{lemma}

\begin{proof}
	This lemma is based on an application of the Burkholder--Davis--Gundy
	inequality. However, a small technical problem crops up when we try to 
	apply that inequality. Namely, that the quantity $\sigma(u_t(\tilde{x}))$
	that appears in \eqref{not-predictable} truly
	depends on the white noise by time $t$. At the same time,
	the stochastic integral of interest is computed over the set
	$\mathbf{B}_\beta(x,t;\varepsilon)\subset[0\,,t]\times\R$,
	and therefore also depends on the white noise by time $t$. In other words,
	\begin{equation}
		\int_{\mathbf{B}_\beta(x,t;\varepsilon)}\sigma(u_t(\tilde{x}))\,\xi(\d s\,\d y)
	\end{equation}
	is \emph{not} a Walsh integral of a predictable process [viewed as a function of
	$t$]; rather it is merely equal to 
	\begin{equation}
		\sigma(u_t(\tilde{x}))\cdot
		\int_{\mathbf{B}_\beta(x,t;\varepsilon)}\xi(\d s\,\d y);
	\end{equation}
	that is, a product of two correlated quantities.
	Thus, we consider first a related quantity $Q_1^{k/2}$, where
	\begin{equation}
		Q_1 :=\left\| \int_{\mathbf{B}_\beta(x,t;\varepsilon) }
			(\nabla_\varepsilon p_{t-s})(y-x)\left[ \sigma(u_s(y))-
			\sigma(u_{t-\beta\varepsilon^\alpha}(\tilde x))\right]
			\xi(\d s\,\d y)\right\|_k^2.
	\end{equation}
	Since $u_{t-\beta\varepsilon^\alpha}(\tilde{x})$ is measurable with
	respect to the white noise of $[0\,,t-\beta\varepsilon^\alpha]\times\R$,
	it is independent of the white noise of $\mathbf{B}_\beta(x,t;\varepsilon)$.
	Therefore, we may apply the Burkholder--Davis--Gundy inequality---%
	see Foondun and Khoshnevisan \cite{FK} for the details of this application---in
	order to see that
	\begin{equation}\begin{split}
		Q_1 &\le c\int_{\mathbf{B}_\beta(x,t;\varepsilon) }
			\left|(\nabla_\varepsilon p_{t-s})(y-x)\right|^2
			\left\|\sigma(u_s(y))-\sigma(u_{t-\beta\varepsilon^\alpha}(\tilde x))\right\|_k^2
			\d s\,\d y\\
		&\le c\lip^2_\sigma\int_{\mathbf{B}_\beta(x,t;\varepsilon) }
			\left|(\nabla_\varepsilon p_{t-s})(y-x)\right|^2
			\left\|u_s(y)-u_{t-\beta\varepsilon^\alpha}(\tilde x)\right\|_k^2
			\d s\,\d y,
	\end{split}\end{equation}
	where $c$ is some constant depending on $k$ and $T$. 
	Let $\tilde x\in[x-\gamma\varepsilon\,,x+\gamma\varepsilon]$.
	Corollary \ref{cor:Modulus:x} and Proposition
	\ref{pro:Modulus:t} together imply the bound
	\begin{align}\notag
		Q_1&\le K\int_{\mathbf{B}_\beta(x,t;\varepsilon) }
			\left|(\nabla_\varepsilon p_{t-s})(y-x)\right|^2
			\left[|\tilde x-y|^{\alpha-1}+|t-\beta\varepsilon^\alpha-s|^{(\alpha-1)/\alpha}
			\right] \d s\,\d y\\
		&\le K\varepsilon^{\alpha-1}[2\gamma+\beta^{1/\alpha}]^{\alpha-1} 
			\int_{\mathbf{B}_\beta(x,t;\varepsilon)}
			\left|(\nabla_\varepsilon p_{t-s})(z)\right|^2 
			\d s\,\d z,
	\end{align}
	where $K$ depends only on $\alpha\,,k$ and $T$. In the preceding, the $\d z$-integral
	ranges over $z\in[x-\varepsilon\gamma\,,x+\varepsilon\gamma]$. 
	If we replace $\mathbf{B}_\beta(x\,,t\,;\varepsilon) $ by $(0\,,t)\times\R$,
	then we obtain the bound
	\begin{equation}
		Q_1 \le K\varepsilon^{\alpha-1}[2\gamma+\beta^{1/\alpha}]^{\alpha-1}\E\left( 
		\left| (\nabla_\varepsilon Z_t)(0)\right|^2\right)
		\le K'\varepsilon^{(\alpha-1)}[2\gamma+\beta^{1/\alpha}]^{\alpha-1},
	\end{equation}
	thanks to Lemma~\ref{(0,t)timesR}. 
	
	Next we estimate the cost of estimating $u_t(\tilde{x})$
	by $u_{t-\beta\varepsilon^\alpha}(\tilde{x})$. Indeed,
	by the H\"older inequality, 
	\begin{align}\notag
		Q_2&:=\left\|\left[\sigma(u_t(\tilde x))-
			\sigma(u_{t-\beta\varepsilon^\alpha}(\tilde x))\right]
			\int_{\mathbf{B}_\beta(x,t;\varepsilon) }(\nabla_\varepsilon p_{t-s})(y-x)
			\xi(\d s\,\d y)\right\|_k^2\\\notag
		&\leq\lip_\sigma^2\left\|u_t(\tilde x) - 
			u_{t-\beta\varepsilon^\alpha}(\tilde x)\right\|_{2k}^2
			\left\|\int_{\mathbf{B}_\beta(x,t;\varepsilon) }(\nabla_\varepsilon p_{t-s})(y-x)
			\xi(\d s\,\d y)\right\|_{2k}^2\\\notag
		&\leq K\beta^{(\alpha-1)/\alpha}\varepsilon^{(\alpha-1)}
			\int_{\mathbf{B}_\beta(x,t;\varepsilon)}
			\left|(\nabla_\varepsilon p_{t-s})(z)\right|^2 
			\d s\,\d z\\
		&\leq K\beta^{(\alpha-1)/\alpha}\varepsilon^{2(\alpha-1)}.
	\end{align}
	Since $\gamma = 1+\beta^{3/2}\le2\beta^{3/2}$ 
	and $\beta^{1/\alpha}\leq \beta^{3/2}$, we can conclude that 
	\begin{align}
		\sqrt{Q_1}+\sqrt{Q_2}\leq c\,\varepsilon^{(\alpha-1)}\beta^{3(\alpha-1)/4}.
	\end{align}
	This and Minkowski's inequality together imply the lemma.
\end{proof}

Now we conclude our first main effort.

\begin{proof}[Proof of Theorem \ref{th:Differentiate}]
	We begin the proof by looking at the following quantity:
	\begin{equation}\label{eq:E(error^k)}
			\sup_{x\in\R}\sup_{t\in[0,T]}
			\| (\nabla_\varepsilon u_t)(x) 
			-\mathfrak{A}_\alpha\sigma(u_t(\tilde x))(\nabla_\varepsilon F)(x)
			\|_k;
		\end{equation}
	where $\tilde x\in[x-\gamma\varepsilon\,,x+\gamma\varepsilon]$.
	Let us first split the preceding expectation in two parts as follows:
	\begin{eqnarray*}
		\| (\nabla_\varepsilon u_t)(x) 
		-\mathfrak{A}_\alpha\sigma(u_t(x))(\nabla_\varepsilon F)(x)\|_k\le I_1+I_2;
	\end{eqnarray*}	
	where 
	\begin{equation}\label{eq:error}
		I_1 := \sup_{x\in\R}\sup_{t\in[0,T]}\left\| (\nabla_\varepsilon u_t)(x)
		-\sigma(u_t(\tilde x))(\nabla_\varepsilon Z_t)(x) \right\|_k;
	\end{equation}	
	and 
	\begin{equation}\label{I_2}
		I_2 :=\sup_{x\in\R}\sup_{t\in [0, T]}\|\sigma(u_t(\tilde x)) (\nabla_\varepsilon Z_t)(x) 
		-\mathfrak{A}_\alpha\sigma(u_t(\tilde x))(\nabla_\varepsilon F)(x)\|_k.
	\end{equation}	
	
	Consider first the quantity $I_1$. The cost of replacing 
	$\sigma(u_t(\tilde x))$ by $\sigma(u_t(x))$ is controlled 
	by the triangle inequality, which yields the bound
	\begin{equation}\label{holder}\begin{split}
		\|(\nabla_\varepsilon Z_t)(x)[\sigma(u_t(\tilde x))&-\sigma(u_t(x))]\|_k\\
		&\leq\|(\nabla_\varepsilon Z_t)(x)\|_{2k}\lip_\sigma\|u_t(\tilde x)-u_t(x)\|_{2k}\\
		&\leq A\varepsilon^{\alpha-1}.
	\end{split}\end{equation}
	[The last inequality follows from Lemma \ref{(0,t)timesR} and Corollary
	\ref{cor:Modulus:x}.]
	Lemma \ref{lem:localize:u:1} and Corollary \ref{cor:localize:u} together
	imply that
	\begin{equation}\label{b4opt}\begin{split}
		&\left\| (\nabla_\varepsilon u_t)(x)
		-\sigma(u_t(x))(\nabla_\varepsilon Z_t)(x) \right\|_k \\
		&\hskip1.7in\le c'\,\varepsilon^{\alpha-1}\beta^{3(\alpha-1)/4} + 
		c'\,\varepsilon^{(\alpha-1)/2}\beta^{-1/4},
	\end{split}\end{equation}
	where $c'$ denotes a finite constant that does not depend on
	the values of $x\in\R$,
	$t\in[0\,,T]$, and $\varepsilon\in(0\,,1)$. 
	Define $\beta := \varepsilon^{-2b}>1$. Since the
	left-hand side of \eqref{b4opt} does not depend on $\beta$, 
	we can optimize the right-hand side over $b>0$, 
	to find that the best bound in~\eqref{b4opt} is attained when
	\begin{equation}\label{b}
		b :=\frac{\alpha-1}{3\alpha-2}.
	\end{equation}
	This particular choice yields
	\begin{equation}\label{eq:error1}
		\left\| (\nabla_\varepsilon u_t)(x)
		-\sigma(u_t(x))(\nabla_\varepsilon Z_t)(x) \right\|_k
		\le c\,\varepsilon^{(\alpha-1+b)/2}.
	\end{equation}	
	Because $(\alpha-1+b)/2<\alpha-1$, we deduce from
	\eqref{holder} and \eqref{eq:error1} the following bound for $I_1$:
	\begin{equation}\label{eq:error}
		I_1\le c\,\varepsilon^{(\alpha-1+b)/2}.
	\end{equation}
	
	In order to bound $I_2$, we appeal to Proposition \ref{pr:fBm} and write
	\begin{equation}
		Z_t(x) = \mathfrak{A}_\alpha
		F(x) + S(x)\qquad(x\in\R),
	\end{equation}
	where $F$ is fBm($(\alpha-1)/2$) and $S$ is a mean-zero Gaussian process
	with $C^\infty$ trajectories. In accord with \eqref{eq:E(dS2)}, there
	exists a finite constant $A$ such that $\|(\nabla_\varepsilon S)(x)\|_2\le A\varepsilon$
	simultaneously for all $x\in\R$ and $\varepsilon>0$. Because the variance of a mean-zero
	Gaussian random variable determines all of its moments, we can find for all $k\in[2\,,\infty)$
	a finite constant $A_k$ such that
	\begin{equation}
		\sup_{x\in\R}\E \left(  | (\nabla_\varepsilon S)(x)  |^k \right) \le A_k
		\varepsilon^k,
	\end{equation}
	for all $\varepsilon>0$.
	Equivalently,
	\begin{equation}
		\sup_{x\in\R}\E\left( \left| (\nabla_\varepsilon Z_t)(x) 
		-\mathfrak{A}_\alpha(\nabla_\varepsilon F)(x)
		\right|^k\right) \le A_k \varepsilon^k,
	\end{equation}
	uniformly for all $\varepsilon>0$. Since $\sup_{t\in[0,T]}\sup_{x\in\R}
	\E(|\sigma(u_t(x))|^2 )<\infty$ for all $T>0$, the Cauchy--Schwarz inequality
	yields the following bound: For all $k\in[2\,,\infty)$ and $T>0$, there exists
	a finite constant $A'$ such that
	$I_2 \le A' \varepsilon$,
	uniformly for all $\varepsilon>0$.  
	Because $(\alpha-1+b)/2<1$, we then deduce from \eqref{eq:error} that
	for all $\varepsilon\in(0\,,1)$,	
	\begin{equation}
		I_1+I_2\leq  A'\,\varepsilon^{(\alpha-1+b)/2}\qquad
		(0<\varepsilon<1).
	\end{equation}
	Equivalently,
	\begin{equation}\label{eq:E(error^k)}
		\sup_{x\in\R}\sup_{t\in[0,T]}
		\E\left( \left| (\nabla_\varepsilon u_t)(x) 
		-\mathfrak{A}_\alpha\sigma(u_t(\tilde x))(\nabla_\varepsilon F)(x)
		\right|^k\right) \le A'\,\varepsilon^{k(\alpha-1+b)/2},
	\end{equation}
	where $\tilde x\in[x-\gamma\varepsilon\,,x+\gamma\varepsilon]$, 
	and the finite constant $A'$ does not depend on $\varepsilon$.
	
	We are now ready to complete the proof of Theorem \ref{th:Differentiate}.
	Set $\tilde x:=x$. Thanks to \eqref{eq:E(error^k)} and
	Chebyshev's inequality, for all $\zeta\in(0\,, b)$ and $\varepsilon\in(0\,,1)$,
	\begin{equation}
		\sup_{x\in\R}
		\P\left\{\left| \frac{(\nabla_\varepsilon u_t)(x)}{(\nabla_\varepsilon F)(x)}	
		- \mathfrak{A}_\alpha\sigma(u_t(x))
		\right|>\frac{\varepsilon^{(\alpha-1+\zeta)/2}}{|(\nabla_\varepsilon F)(x)|}\right\}
		\le A'\varepsilon^{k(b-\zeta)/2}.
	\end{equation}
	The preceding makes sense because $(\nabla_\varepsilon F)(x)$ is a.s.\  non zero,
	as it is a centered Gaussian random variable. This concludes the proof because
	it is easy to see that $(\nabla_\varepsilon F)(x)$ has the same distribution as
	$F(\varepsilon)$, regardless of the value of $x\in\R$,
	and the latter random variable has the same
	law as $\varepsilon^{(\alpha-1)/2}$ times a  standard Gaussian random variable $\mathcal N$, whence
	\begin{equation}
		\lim_{\varepsilon\downarrow 0}
		\P\left\{\frac{\varepsilon^{(\alpha-1+\zeta)/2}}{|(\nabla_\varepsilon F)(x)|}
		>\lambda\right\} = \lim_{\varepsilon\downarrow 0}
		\P\left(|\mathcal N|<\frac{\varepsilon^{\zeta/2}}{\lambda}\right)
		= 0,
	\end{equation}
	uniformly in $x\in\R$ and for all $\lambda>0$.
\end{proof}

\section{Proof of Corollary \ref{cor:LIL}}
Fix $x\in \R$ and $t>0$. We shall prove the following
``strong approximation'' result: 
Choose and fix $q\in(0\,,b)$, where $b$ is defined by~\eqref{b}. Then, with probability one,
\begin{equation}\label{eq:SA}
	\left| (\nabla_\varepsilon u_t)(x) - \mathfrak{A}_\alpha \sigma(u_t(x))(
	\nabla_\varepsilon F)(x)\right|= o\left( \varepsilon^{(\alpha-1+q)/2}\right),
\end{equation}
as $\varepsilon\downarrow 0$.

Corollary \eqref{cor:LIL}
is then a ready consequence of \eqref{eq:SA} and the law of the iterated logarithm for
fractional Brownian motion, which itself follows fairly readily
from Theorem 1.1 of \cite{QuallsWatanabe}.

Let us define a stochastic process $G$ via
\begin{align}
	G(\varepsilon)=u_t(x-\varepsilon)-\mathfrak{A}_\alpha\sigma(u_t(x))F(x-\varepsilon)
	\qquad(0\le \varepsilon\le x).
\end{align}
For all $x\ge \nu\ge \varepsilon\geq0$, we let $r:=\nu-\varepsilon$ and see that
\begin{align}
	|G(\varepsilon)-G(\nu)|&=
	\big|\nabla_{r}u_t(x-\varepsilon)-
	\mathfrak{A}_\alpha\sigma(u_t(x))\nabla_{r}F(x-\varepsilon)\big|.
\end{align}
Therefore, if we let $z:=x-\varepsilon$, then
\begin{equation}\begin{split}
	\|G(\varepsilon)-G(\nu)\|_k &=
		\big\|\nabla_{r}u_t(z)-
		\mathfrak{A}_\alpha\sigma(u_t(z+\varepsilon))\nabla_{r}F(z)\big\|_k\\
	&\leq A'r^{(\alpha-1+b)/2},
\end{split}\end{equation}
thanks to \eqref{eq:E(error^k)}. By the 
Kolmogorov's continuity theorem, $G$ is locally a H\"older-continuous process 
with any index $a\in(0\,,\frac12(\alpha-1)+b-k^{-1})$. 
Recall from~\eqref{b} that $b>0$ for all $\alpha>1$, and that $k$ is arbitrary. 
Therefore, by choosing $k$ sufficiently large, we find $q:= b-k^{-1}>0$.
Next we combine the law of the iterated logarithm for the 
fBm $F$ with~\eqref{eq:SA} to finish the proof.\qed

\section{Proof of Corollary \ref{a.s.}}
Define, for all $s>0$,
\begin{equation}\label{X_s}
	X_s:=\int_0^s\left| \frac{(\nabla_\varepsilon u_t)(x)}{(\nabla_\varepsilon F)(x)}
	 - \mathfrak{A}_\alpha \sigma(u_t(x))\right|\d\varepsilon.
\end{equation}
Choose and fix $q\in(0\,,b)$, where $b$ is defined by~\eqref{b}. 
Then, Minkowski's inequality and \eqref{eq:SA} together imply that,
with probability one,
\begin{equation}
	X_s =o\left(\int_0^s\frac{\varepsilon^{(\alpha-1+q)/2}}{%
	|(\nabla_\varepsilon F)(x)|}
	\,\d\varepsilon\right)\qquad(s\downarrow0).
\end{equation}
Since $\varepsilon^{(\alpha-1)/2}\nabla_\varepsilon F$ has a 
standard normal distribution, Minkowski's inequality
guarantees that for every $r\in(0\,,1)$ and $s>0$,
\begin{equation}
	\E(|X_s|^r)
	\leq \E(\mathcal{N}^{-r})\left|\int_0^s\varepsilon^{q/2}
	\,\d\varepsilon\right|^r
	= c s^{r(2+q)/2},
\end{equation}	
where $c\in(0\,,\infty)$ depends only on $r$, and
$\mathcal{N}$ is a random variable with the standard normal distribution. 
Therefore, by the Chebyshev inequality, 
\begin{equation}
	\P\left\{|X_{2^{-n}}|>\lambda 2^{-n\tau}\right\}
	\leq c2^{-nr(1+(q/2)-\tau)}\lambda^{-r},
\end{equation}
for all $\lambda,\tau>0$ and integers $n\ge 0$.
If we let $\tau<1+(q/2)$, then we can see from the Borel-Cantelli lemma that
\begin{equation}\label{dyadicLim}
	 \lim_{n\to\infty}2^{n\tau}X_{2^{-n}}=0
	 \qquad\text{a.s.}
\end{equation}
If $s$ is a real number in $[2^{-n-1},\,2^{-n}]$
for some integer $n\ge 0$, then $s^{-\tau}X_s \le 2^{\tau(n+1)}X_{2^{-n}}$.
Therefore,
\eqref{dyadicLim} implies that $X_s=o(s^\tau)$ a.s.,
as was claimed.
\qed

\section{Proof of Corollary \ref{cor:CLT}} 
Throughout this proof we write
\begin{equation}
	\beta := \beta(\varepsilon) := \log(1/\varepsilon),
\end{equation}
and
\begin{equation}
	\mathcal{I}_\varepsilon(x\,,t) := \int_{\mathbf{B}_\beta(x,t;\varepsilon)}
	(\nabla_\varepsilon p_{t-s})(y-x)\,\xi(\d s\,\d y),
\end{equation}
for the sake of notational simplicity.

The proof of Corollary \ref{cor:CLT} hinges on
the work that has been developed so far, together with a rather crude estimate
on the temporal modulus of continuity of $u$, namely~\ref{pro:Modulus:t}.
See also \cite[Theorem 2]{DebbiDozzi}. 
We  need to know only that the following holds for every $k\in[2\,,\infty)$ and $T>0$ fixed:
Uniformly for all $\varepsilon\in(0\,,1)$,
\begin{equation}
	\left\| \left\{\sigma(u_t(x))-\sigma(u_{t- \varepsilon^\alpha}(x))\right\}\cdot
	\mathcal{I}_\varepsilon(x\,,t)\right\|_k
	\le \text{const}\cdot \varepsilon^{(\alpha-1)/2}
	\left\| \mathcal{I}_\varepsilon(x\,,t)\right\|_{2k},\label{that}
\end{equation}
thanks to the Cauchy--Schwarz inequality
and the already-mentioned fact that $\sup_{x\in\R}\sup_{t\in[0,T]}
\E(|\sigma(u_t(x))|^k)<\infty$. We apply Proposition \ref{pr:localize:Z} next in order
to see that
\begin{equation}\begin{split}
	\left\| \mathcal{I}_\varepsilon(x\,,t)\right\|_{2k}
		&\le\textnormal{const}\cdot\sqrt{\frac{\varepsilon^{\alpha-1}}{\log(1/\varepsilon)}}
		+ \left\|(\nabla_\varepsilon Z_t)(x)\right\|_k\\
	&\le \text{const}\cdot\varepsilon^{(\alpha-1)/2}.
	\label{that2}
\end{split}\end{equation}
In the last two inequalities, we have used Lemma \ref{(0,t)timesR} and
also the fact that the $L^2$-norm of a Gaussian random variable
determines all of its $L^k$-norms. This can be combined with \eqref{that},
Corollary \ref{cor:localize:u},
and Lemma \ref{lem:localize:u:1} in order to show that
\begin{equation}\label{eq:localize:u:2}
	\left\| (\nabla_\varepsilon u_t)(x) - \sigma(u_{t-\beta\varepsilon^\alpha}(x))
	\mathcal{I}_\varepsilon(x\,,t)\right\|_k
	\le \text{const}\cdot \frac{\varepsilon^{(\alpha-1)/2}}{\left|\log(1/\varepsilon)\right|^{1/4}},
\end{equation}
uniformly for $x\in\R$, $t\in[0\,,T]$, and $\varepsilon\in(0\,,1)$.
Consequently, we see that $\varepsilon^{-(\alpha-1)/2}(\nabla_\varepsilon u_t)(x)$
and $\varepsilon^{-(\alpha-1)/2}\sigma(u_{t-\beta\varepsilon^\alpha}(x))
\mathcal{I}_\varepsilon(x\,,t)$ have the same asymptotic
behavior. 

Next let us observe that, because of the independence properties of space-time 
white noise, $\sigma(u_{t-\beta\varepsilon^\alpha}(x))$ is independent of
$\varepsilon^{-(\alpha-1)/2}\mathcal{I}_\varepsilon(x\,,t)$, and converges to
$\sigma(u_t(x))$ a.s.\ as $\varepsilon\downarrow 0$, by continuity. [The requisite
continuity is assured by Proposition~\ref{pro:Modulus:t} and the Kolmogorov continuity theorem.]
Therefore, the asymptotic behavior of $\varepsilon^{-(\alpha-1)/2}
(\nabla_\varepsilon p_{t-s})(y-x)$ is the same as the asymptotic behavior of
$\sigma(\tilde{u}_t(x))\times \varepsilon^{-(\alpha-1)/2}\mathcal{I}_\varepsilon
(x\,,t)$, where $\tilde{u}$ is independent
of $\mathcal{I}_\varepsilon(x\,,t)$ and has the same law as $u$. Proposition
\ref{pr:localize:Z} assures that the asymptotic behavior of
$\varepsilon^{-(\alpha-1)/2}\mathcal{I}_\varepsilon(x\,,t)$ is
the same as that of $\varepsilon^{-(\alpha-1)/2}(\nabla_\varepsilon Z_t)(x)$.
According to Proposition \ref{pr:fBm},
\begin{equation}
	\varepsilon^{-(\alpha-1)/2}(\nabla_\varepsilon Z_t)(x)=\mathfrak{A}_\alpha
	\varepsilon^{-(\alpha-1)/2}(\nabla_\varepsilon F)(x) + O\left( \varepsilon^{(3-\alpha)/2}\right)
	\qquad\text{a.s.},
\end{equation}
and Corollary \ref{cor:CLT} follows from the defining property of fBm.
In this case, that property implies that $\varepsilon^{-(\alpha-1)/2}
(\nabla_\varepsilon F)(x)$ has a standard normal law.\qed

\section{Proof of Corollary \ref{cor:VAR}}
Throughout this proof, we use the short-hand notation:
\begin{equation}\label{epsbeta}
	\varepsilon := \varepsilon(n) := 2^{-n}
	\quad\text{and}\quad
	\beta :=\beta(n) := n^{16/(\alpha-1)}.
\end{equation}

We wish to prove that
\begin{equation}\begin{split}
	&\lim_{n\to\infty} \sum_{a\le j\varepsilon \le b} 
		\phi( u_t (j \varepsilon)) \left| (\nabla_\varepsilon u_t)((j+1)\varepsilon)
		\right|^{2/(\alpha-1)}\\
	&\hskip1.7in
		=\mathfrak{B}_\alpha \int_a^b \phi(u_t(x))
		\left[ \sigma(u_t(x))\right]^{2/(\alpha-1)}\d x,
\end{split}\end{equation}
a.s.\ and in $L^2(\P)$. This is done via a series of 1-step reductions.
First, let us apply the Cauchy-Schwarz inequality to see that
\begin{equation}\begin{split}
	&\left\| \left| \phi( u_t (j \varepsilon)) - \phi(
		u_{t-\beta\varepsilon^\alpha}(j\varepsilon))\right|
		\times \left| (\nabla_\varepsilon u_t)((j+1)\varepsilon)
		\right|^{2/(\alpha-1)}\right\|_k\\
	&\le \sup_{x\in\R} \left\| \phi( u_t (x)) - \phi(
		u_{t-\beta\varepsilon^\alpha}(x)\right\|_{2k}\times
		\sup_{y\in\R}\left\| (\nabla_\varepsilon u_t)(y)
		\right\|_{4k/(\alpha-1)}^{2/(\alpha-1)}\\
	&\le\text{const}\cdot n^{16\delta/(\alpha-1)}2^{-n(1+\delta)},
\end{split}\end{equation}
thanks to \eqref{pro:Modulus:t} and Corollary \ref{cor:Modulus:x}.
Therefore, the Borel--Cantelli reduces our problem to verifying that,
almost surely and in $L^2(\P)$,
\begin{equation}\begin{split}
	&\lim_{n\to\infty} \sum_{a\le j\varepsilon \le b} 
		\phi( u_{t-\beta\varepsilon^\alpha} (j \varepsilon)) \left| (\nabla_\varepsilon u_t)((j+1)\varepsilon)
		\right|^{2/(\alpha-1)}\\
	&\hskip1.7in
		=\mathfrak{B}_\alpha \int_a^b \phi(u_t(x))
		\left[ \sigma(u_t(x))\right]^{2/(\alpha-1)}\d x.
\end{split}\end{equation}
Similarly, we apply \eqref{eq:localize:u:2} in order to reduce the problem further
to one about showing that, almost surely and in $L^2(\P)$,
\begin{equation}\begin{split}
	&\lim_{n\to\infty} \sum_{a\le j\varepsilon \le b} 
		\phi( u_{t-\beta\varepsilon^\alpha} (j \varepsilon)) 
		\left|\sigma(u_{t-\beta\varepsilon^\alpha} (j \varepsilon)) \right|^{2/(\alpha-1)}
		\left| \mathcal{I}_\varepsilon(j\varepsilon\,;t)\right|^{2/(\alpha-1)}\\
	&\hskip1.7in
		=\mathfrak{B}_\alpha \int_a^b \phi(u_t(x))
		\left[ \sigma(u_t(x))\right]^{2/(\alpha-1)}\d x.
\end{split}\end{equation}

Let us define the function
\begin{equation}
	\mathcal{H}(x) := \phi(x)\left|\sigma(x) \right|^{2/(\alpha-1)}
	\qquad(x\in\R),
\end{equation}
in order to simplify some of the typography. We emphasize once
again that:
\textbf{(i)} $\sup_{x\in\R}\sup_{t\in[0,T]}\E(|\mathcal{H}(u_t(x))|^k)<\infty$
for all $T>0$; and \textbf{(ii)} Our goal is to prove that,
a.s.\ and in $L^k(\P)$,
\begin{equation}\label{H:0}
	\lim_{n\to\infty} \sum_{a\le j\varepsilon \le b} 
	\mathcal{H}( u_{t-\beta\varepsilon^\alpha} (j \varepsilon))
	\left| \mathcal{I}_\varepsilon(j\varepsilon\,;t)\right|^{2/(\alpha-1)}
	=\mathfrak{B}_\alpha \int_a^b \mathcal{H}(u_t(x))\,\d x.
\end{equation}
Because $\mathcal{I}_\varepsilon(j\varepsilon\,;t)$ is independent of
$\{u_{t-\beta\varepsilon^\alpha}(i\varepsilon)\}_{i\le j}$,
\textbf{(i)} implies that
\begin{align}
	&\text{Var}\sum_{a\le j\varepsilon \le b} 
		\mathcal{H}( u_{t-\beta\varepsilon^\alpha} (j \varepsilon))
		\left| \mathcal{I}_\varepsilon(j\varepsilon\,;t)\right|^{2/(\alpha-1)}\\\notag
	&\le\text{const}\cdot\sum_{a\le j\varepsilon\le b}
		\text{Var}\left(\left| \mathcal{I}_\varepsilon(j\varepsilon\,;t)
		\right|^{2/(\alpha-1)}\right) \le\text{const}\cdot\sum_{a\le j\varepsilon\le b}
		\E\left(\left| \mathcal{I}_\varepsilon(j\varepsilon\,;t)
		\right|^{4/(\alpha-1)}\right),
\end{align}
uniformly for all $t\in[0\,,T]$ and $n\ge 1$ [whence also $\varepsilon\in(0\,,1)$
and $\beta>1$, as defined in \eqref{epsbeta}]. The
right-most quantity is $O(n^{-2})$; this is shown very much as \eqref{that2}
was, but the parameter $\beta$ has to now be adjusted.
In this way we reduce our problem \eqref{H:0}---%
thanks to the Borel--Cantelli lemma---to one
about proving  that, a.s.\ and in $L^2(\P)$,
\begin{equation}\label{H:1}
	\lim_{n\to\infty} \sum_{a\le j\varepsilon \le b} 
	\mathcal{H}( u_{t-\beta\varepsilon^\alpha} (j \varepsilon))
	\E\left(\left| \mathcal{I}_\varepsilon(j\varepsilon\,;t)\right|^{2/(\alpha-1)}\right)
	=\mathfrak{B}_\alpha \int_a^b \mathcal{H}(u_t(x))\,\d x.
\end{equation}
In accord with Proposition \ref{pr:localize:Z},
\begin{equation}
	\E\left( \left| (\nabla_\varepsilon Z_t)(x) - \mathcal{I}_\varepsilon(x\,;t)\right|^{2/
	(\alpha-1)}\right) \le\text{const}\cdot  n^{-2}2^{-n},
\end{equation}
uniformly for all $t\in[0\,,T]$ and $n\ge 1$.
Therefore, we may appeal to Corollary \ref{cor:Mean:Linear} in order to see
that
\begin{equation}
	\E\left(\left| \mathcal{I}_\varepsilon(j\varepsilon\,;t)\right|^{2/(\alpha-1)}\right) 
	= \mathfrak{B}_\alpha 2^{-n}(1+o(1))
	\qquad(n\to\infty).
\end{equation}
This and the continuity of the random function $(t\,,x)\mapsto u_t(x)$ together
show that \eqref{H:1} holds, thanks to a Riemann--sum approximation.
This completes our proof of the corollary.

\section{Comments on variance stabilization}\label{sec:VS}

Corollary \ref{cor:CLT} implies that if $\varepsilon$ is small, then
$(\nabla_\varepsilon u_t)(x)$ behaves as $\varepsilon^{(\alpha-1)/2}$
times $\sigma(u_t(x))Z$ where $Z$ is an independent standard normal 
random variable.
The aim of this section is to show how one
can apply a conditional form of what statisticians call a
``variance stabilizing transformation''  in order to simplify the preceding
into a bona fide central limit theorem with a Gaussian limit.

Let us first recall a classical definition that has also been recently adapted
to the potential theory of stochastic PDEs
\cite{DalangNualart}.
 
\begin{definition}
	A Borel set $E\subset\R$ is \emph{polar} for $u$ if
	\begin{equation}
		\P\left\{ u_t(x) \in E \text{ for some $t\ge 0$ and $x\in\R$}\right\}=0.
	\end{equation}
\end{definition}
Our definition of polar sets is slightly different from its progenitor in
probabilistic potential theory: Note that $t=0$ {\it is} included.
\begin{lemma}
	If $\sigma^{-1}\{0\}:=\{z\in\R:\ \sigma(z)=0\}$ 
	is polar for $u$, then the following is a continuous
	random field indexed by $\R_+\times\R$:
	\begin{equation}\label{X}
		X_t(x) := \int_{u_t(0)}^{u_t(x)}\frac{\d y}{\sigma(y)}
		\qquad(t\ge 0,\,x\in\R).
	\end{equation}
\end{lemma}

\begin{proof}
	Recall that $u$ is a continuous random function of $(t\,,x)\in\R_+\times\R$.
	Define $J(t\,,x)$ as the closed random subinterval of $\R$ whose endpoints are
	$u_t(0)$ and $u_t(x)$.
	Since $\sigma^{-1}\{0\}$ is polar for $u$, there is a $\P$-null set off which
	$J(t\,,x)\cap\sigma^{-1}\{0\}=\varnothing$ simultaneously for all $x\in\R$
	and $t\ge 0$; this holds because of the mean value theorem.  Therefore,
	the continuity of $\sigma$ ensures that, in addition, $\inf_{y\in J(t\,,x)}|\sigma(y)|>0$
	a.s. The rest of the proof is easy.
\end{proof}

Let $X$ denote the random field that was defined in \eqref{X}. If
$\sigma^{-1}\{0\}$ were polar for $u$, then $X$ would be well defined and
\begin{equation}
	X_t(x) - X_t(x-\varepsilon) = \frac{u_t(x)-u_t(x-\varepsilon)}{\sigma(u_t(x))}+
	O\left(\left| u_t(x)-u_t(x-\varepsilon)\right|^2\right),
\end{equation}
uniformly for every $x\in[a\,,b]$ and $t\in[0\,,T]$. This observation
and Theorem \ref{th:Differentiate} together yield the following.

\begin{theorem}\label{th:X:Differentiate}
	If $\sigma^{-1}\{0\}$ is polar for $u$,
	then for all $t>0$ and $x\in\R$: (i) 
	\begin{equation}
		\lim_{\varepsilon\downarrow0}\sup_{y\in\R}\P\left\{\left|
		\frac{X_t(y)-X_t(y-\varepsilon)}{F(y)-F(y-\varepsilon)}-\mathfrak{A}_\alpha
		\right|>\lambda\right\}=0\qquad\text{for all $\lambda>0$};
	\end{equation}
	(ii)	With probability one,
	\begin{equation}
		\limsup_{\varepsilon\downarrow0}\frac{X_t(x)-X_t(x-\varepsilon)}{\sqrt{2
		\varepsilon^{\alpha-1}\log\log(1/\varepsilon)}}= - \liminf_{\varepsilon\downarrow0}
		\frac{X_t(x)-X_t(x-\varepsilon)}{\sqrt{2
		\varepsilon^{\alpha-1}\log\log(1/\varepsilon)}}=\mathfrak{A}_\alpha;
	\end{equation}
	(iii) For all $a\in\R$,
	\begin{equation}
		\lim_{\varepsilon\downarrow0}\P\left\{
		\frac{X_t(x)-X_t(x-\varepsilon)}{\varepsilon^{(\alpha-1)/2}}\le a\right\}
		=\P\left\{ \mathcal{N} \le a/\mathfrak{A}_\alpha\right\},
	\end{equation}
	where $\mathcal{N}$ denotes a standard Gaussian random variables; and
	(iv) For all Lipschitz-continuous functions $\varphi:\R\to\R$ and $b>a$ and $t>0$,
	all non random, we have
	\begin{equation}\begin{split}
		&\sum_{a2^n\le j\le 2^nb}\varphi(X_t(j2^{-n}))\left[
			X_t((j+1)2^{-n}) - X_t(j2^{-n})\right]^{2/(\alpha-1)}\\
		&\hskip2.5in=\mathfrak{B}_\alpha\int_a^b\varphi(X_t(x))\,\d x,
	\end{split}\end{equation}
	almost surely and in $L^2(\P)$.
\end{theorem}

Part  (iv) of the preceding requires a real-variable argument that we leave
to the interested reader. The rest are immediate corollaries of Theorem \ref{th:Differentiate}.

For an interesting example, let us consider the ``parabolic Anderson model 
for the Laplacian, driven by space-time white noise,''
$$
	\frac{\partial}{\partial t} u_t(x) = \frac{\partial^2}{\partial x^2}u_t(x) + u_t(x)\xi.
	\eqno{(\text{PAM})}
$$
 That is, (PAM) is the specialization of
the stochastic heat equation (SHE) to the case that 
$\alpha=2$ and $\sigma(x):= x$.
In this case, Mueller (see \cite{Mueller} for a related result,
and \cite[Theorem 5.1, p.\ 130]{Minicourse} for the one that is applicable here) 
has proved that, if in addition $u_0(x)>0$ for all $x\in\R$, then
$\sigma^{-1}\{0\}=\{0\}$ is polar for $u$. The process $X$ is thus well defined,
and is given by
\begin{equation}
	X_t(x) = h_t(x) - h_t(0)\qquad(t\ge 0,\,x\in\R),
\end{equation}
where
\begin{equation}
	h_t(x) := \log u_t(x)\qquad(t\ge 0,\,x\in\R)
\end{equation}
is the so-called Hopf--Cole solution to the KPZ equation \cite{KPZ},
$$
	\frac{\partial}{\partial t} h_t(x) = \frac{\partial^2}{\partial x^2} h_t(x) 
	+\left( \frac{\partial}{\partial x}h_t(x)\right)^2 + \xi.
	\eqno{(\text{KPZ})}
$$
[This is an ill-posed stochastic PDE.]
Since fBm($\nicefrac12$) is standard Brownian motion,
$\mathfrak{A}_2=1/\sqrt 2$, and $\mathfrak{B}_2=1/2$
[see \eqref{A(alpha)} and \eqref{B(alpha)}], we arrive at the following
ready consequence of Theorem \ref{th:X:Differentiate}.

\begin{corollary}\label{cor:KPZ}
	Consider the Hopf--Cole solution $h$ to \textnormal{(KPZ)}, 
	subject to an initial profile $h_0:\R\to\R$ that is non random, 
	uniformly bounded from above,
	and Lipschitz continuous. Then, $h$
	satisfies the following for all $t>0$ and $x\in\R$:
	(i) There exists a Brownian motion $\{B(y)\}_{y\in\R}$ such that
		\begin{equation}
			\lim_{\varepsilon\downarrow0}\sup_{y\in\R}\P\left\{\left|
			\frac{h_t(y)-h_t(y-\varepsilon)}{B(y)-B(y-\varepsilon)}- \frac{1}{\sqrt 2}
			\right|>\lambda\right\}=0\qquad\text{for all $\lambda>0$};
		\end{equation}
	\begin{enumerate}
		\item[(ii)] With probability one,
		\begin{equation}
			\limsup_{\varepsilon\downarrow0}\frac{h_t(x)-h_t(x-\varepsilon)}{%
			\sqrt{\varepsilon\log\log(1/\varepsilon)}}= - \liminf_{\varepsilon\downarrow0}
			\frac{h_t(x)-h_t(x-\varepsilon)}{\sqrt{
			\varepsilon\log\log(1/\varepsilon)}}=1;
		\end{equation}
		\item[(iii)] For all $a\in\R$,
		\begin{equation}
			\lim_{\varepsilon\downarrow0}\P\left\{
			\frac{h_t(x)-h_t(x-\varepsilon)}{\sqrt{\varepsilon/2}}\le a\right\}
			=\frac{1}{\sqrt{2\pi}}\int_{-\infty}^a\exp(-x^2/2)\,\d x;
		\end{equation}
		\item[(iv)] For all Lipschitz-continuous functions $\varphi:\R\to\R$ and $b>a$ and $t>0$,
		all non random, the following holds almost surely and in $L^2(\P)$:
		\begin{align}
			&\sum_{a2^n\le j\le 2^nb}\varphi(h_t(j2^{-n}))\left[
				h_t((j+1)2^{-n}) - h_t(j2^{-n})\right]^2\\\notag
			&\hskip3in= \frac12\int_a^b\varphi(h_t(x))\,\d x.
		\end{align}
	\end{enumerate}
\end{corollary}
\begin{small}

\noindent\\[2mm]

\noindent\textbf{Mohammud Foondun}.
\noindent School of Mathematics, Loughborough University, Leicestershire, UK,
LE11 3TU, {\tt m.i.foondun@lboro.ac.uk}\\

\noindent\textbf{D. Khoshnevisan \&\ P. Mahboubi}.
\noindent Department of Mathematics, The University of Utah,
155 South 1400 East, JWB 233, Salt Lake City, Utah 84105--0090,
{\tt davar@math.utah.edu} \&\ {\tt pejman@math.ucla.edu}

\end{small}

\end{document}